%% file: main.tex
\documentclass[a4paper,11pt]{article}
\usepackage[utf8]{inputenc}
\usepackage[T1]{fontenc}
\usepackage[english]{babel}
\usepackage{amsmath,amssymb,amsthm}
\usepackage[short]{optidef}	
\usepackage{tikz}
\usepackage{subcaption}
\usetikzlibrary{positioning}
\usepackage{soul}
\usepackage{tcolorbox}
\usepackage[shortlabels]{enumitem} 
\usepackage{comment}
\usepackage{standalone}
\usepackage{mathtools}
\usepackage{fullpage}
\usepackage{float}
\usepackage{mathdots}
\usepackage{tikz}
\usepackage[subrefformat=parens,labelformat=parens]{subcaption}
\graphicspath{{./Figuras/}} 
\usetikzlibrary{graphs,graphs.standard}
\usetikzlibrary{decorations.pathmorphing,decorations.pathreplacing,calc}
\usetikzlibrary{decorations.pathmorphing}
\usetikzlibrary{positioning,chains,fit,shapes,calc}

\tikzstyle{selected edge} = [draw,line width=5pt,-,red!50]
\tikzstyle{vertex}=[circle,draw,fill=white,minimum size=10pt,inner sep=2pt]
\tikzstyle{black vertex}=[circle,draw,fill=black,minimum size=10pt,inner sep=2pt]
\tikzstyle{red vertex}=[circle,draw,fill=red!90,minimum size=10pt,inner sep=2pt]
\tikzstyle{gray vertex}=[circle,draw,fill=gray!90,minimum size=20pt,inner sep=1pt]
\tikzstyle{trivial}=[circle,draw,fill=black,minimum size=7pt,inner sep=1pt]
\tikzstyle{dummy}=[circle,minimum size=5pt,inner sep=0pt]
\tikzstyle{red edge} = [draw,very thick,-,red!80,decorate,decoration={snake, amplitude=.8pt, segment length=4pt}]
\tikzstyle{purple edge} = [draw,very thick,-,magenta]
\tikzstyle{blue edge} = [draw,very thick,-,blue!60]
\tikzstyle{green edge} = [draw,very thick,-,green!90]
\tikzstyle{brown edge} = [draw,thick,-,brown!90]
\tikzstyle{black edge} = [draw,thick,-,black!90]
\tikzstyle{dot} = [circle,inner sep=1pt,fill,name=#1]
\tikzstyle{extended line} = [shorten >=-#1,shorten <=-#1]
\usetikzlibrary{arrows.meta}

\usepackage{hyperref}
\usepackage{thmtools}
\usepackage{thm-restate}
\tikzstyle{line} = [draw, thick]

\usetikzlibrary{decorations.pathmorphing,decorations.pathreplacing}

\definecolor{myblue}{RGB}{80,80,160}
\definecolor{mygreen}{RGB}{80,160,80}
\definecolor{myred}{RGB}{255,0,0}
\definecolor{mybrown}{RGB}{139,69,19}

\newtheorem{theorem}             {Theorem}
\newtheorem{lemma}     	[theorem] {Lemma}        
\newtheorem{conjecture}	[theorem] {Conjecture}

\newtheorem{observation}[theorem] {Observation}

\newtheorem{claim}{Claim}

\newtheoremstyle{case}{}{}{}{}{\bfseries}{:}{ }{}
\theoremstyle{case}

\numberwithin{subcase}{case}

\begin{document}

\title{Improved Domination–Packing Bounds in Claw-Free Cubic Graphs and Unit Disk Graphs}

\author{Kaustav Paul\thanks{ Tel Aviv Univesity, Israel.\\
E-mail: {\tt  imkaustav7@gmail.com}} \and Juan Gutierrez\thanks{Faculty of Computing, University of Engineering and Technology, Peru.\\E-mail: {\tt jgutierreza@utec.edu.pe}}}

\author{Juan Gutierrez\thanks{Faculty of Computing, University of Engineering and Technology, Peru. E-mail: {\tt jgutierreza@utec.edu.pe}. Partially supported by Fondo semilla UTEC 2025.}
\and Kaustav Paul\thanks{ Tel Aviv Univesity, Israel.
E-mail: {\tt  imkaustav7@gmail.com}}}

\date{}
\maketitle

\begin{center}
\footnotesize

\bigskip

Departamento de Ciencia de la Computación\\ 
Universidad de Ingeniería y Tecnología (UTEC), Lima, Perú\\
E-mail: \texttt{jgutierreza@utec.edu.pe}

Departamento de Ciencia de la Computación\\ 
Universidad de Ingeniería y Tecnología (UFABC), Lima, Perú\\
E-mail: \texttt{jgutierreza@utec.edu.pe}

Tel Aviv Univesity, Israel\\
E-mail: \texttt{imkaustav7@gmail.com}

\end{center}

\maketitle

\begin{abstract}
Given a graph $G$, the domination number $\gamma(G)$ is the minimum cardinality of a dominating set in $G$, and the packing number $\rho(G)$ is the maximum cardinality of a set of vertices that are pairwise at distance at least $3$. The ratio between these parameters has been widely studied in several graph classes. It is known that $\gamma(G) \le 2\rho(G)$ for claw-free subcubic graphs, up to finitely many exceptions, and that $\gamma(G) \le 32\rho(G)$ for unit disk graphs.
In this paper, we improve the latter bound by showing that $\gamma(G) \le 16\rho(G)$ for a unit disk graph $G$. For the former bound, we show that it can be improved in the cubic bridgeless setting; more precisely, every bridgeless claw-free cubic graph $G$ satisfies $\gamma(G) \le \frac{7}{4}\rho(G) + \frac{5}{6}$.
These results are not tight. In fact, we give example of an infinite family of bridgeless cubic graphs $G$ with $\gamma(G) = 5\rho(G)/4$ and an infnite family of unit disk graphs $G$
in which $\gamma(G) = 3\rho(G)$.
\end{abstract}

\section{Introduction and Preliminaries}

\input{intro.tex}



\section{Unit Disk graphs}

\input{UDG}

\section{Claw free cubic graphs}
\input{claw-free-cubic.tex}
\input{claw-free-cubic-family.tex}



\bibliographystyle{abbrv}
\bibliography{refer}
\end{document}

%% file: intro.tex
Let $G=(V,E)$ be a graph. For any vertex $u \in V$, we denote by $N(u)$ the neighborhood of $v$, and set $N[u]=\{u\} \cup N(u)$. Also, $N_i(u)$ is the set of vertices at distance $i$ from $u$ and we put $N_{\leq i}(u) = N[u] \cup \cdots \cup N_i(u)$, that is, the set of vertices at distance at most $i$ from $u$.
These definitions extend naturally to sets $S\subseteq V$. That is, $N(S) = \bigcup_{u \in S} N(u), N[S]=S \cup N(S), 
N_i(S) = \bigcup_{u \in S} N_i(u)$, and
$N_{\leq i}(S) = N[S] \cup \cdots \cup N_i(S)$.

A set $D \subseteq V$ is a \emph{dominating set} if every vertex of $G$ is either in $D$ or adjacent to a vertex in $D$. Equivalently, if $N[D]=V(G)$. A set $P \subseteq V$ is a \emph{packing} if $N[u] \cap N[v] = \emptyset$ for any distinct $u,v \in P$. Equivalently, if every pair of vertices in $P$ are at distance at least 3.  The \emph{domination number} $\gamma(G)$ is the minimum size of a dominating set, and the \emph{packing number} $\rho(G)$ is the maximum size of a packing.

The relationship between $\gamma(G)$ and $\rho(G)$ has been extensively studied. The inequality $\gamma(G) \geq \rho(G)$ holds and Henning et al.\ showed that $\gamma(G) \leq  \Delta(G) \rho(G)$, where $\Delta(G)$ is the maximum degree of $G$ \cite{Henning2011}.
It is known that $\gamma(G)= \rho(G)$ holds if $G$ is a tree \cite{MeirM75}, a strongly chordal graph \cite{Farber84}, or a dually chordal graph \cite{BrandstadtCD98}.
It is also known that $\gamma(G) \leq 2\rho(G)$ holds if $G$ is a cactus graph \cite{LowensteinRR13} or a biconvex graph \cite{Gomez2025}.
Regarding planar graphs, it was shown that
$\gamma(G) \leq 3 \rho(G)$ for any maximal outerplanar $G$. Such bound was later improved to $2$ \cite{bonamy2025}.
For arbitrary planar graphs, it was shown that $\gamma(G) \leq 10 \rho(G)$, a bound later improved to $7$ \cite{ducz2026}. 

It was recently shown that $\gamma(G) \le 32\rho(G)$ for unit disk graphs \cite{bonamy2025}. 
In this paper we improve this result as follows

\begin{restatable}{theorem}{UnitDiskTheorem}\label{th:UDG_main}
Given a unit disk graph $G$, $\gamma(G)\leq 16\cdot \rho(G)$.
\end{restatable}

The bound in Theorem \ref{th:UDG_main} is not tight, however we find an infinite family of graphs that attain a factor of $3$.

\begin{restatable}{theorem}{UnitDiskFamily}\label{theorem:UDG_family}
 There exists an infinite family of unit disk graphs $G$ with $\gamma(G)= 3 \rho(G)$.
\end{restatable}

This motivates the following conjecture.

\begin{restatable}{conjecture}{UnitDiskConjecture}\label{conj:UDG_conjecture}
For every unit disk graph \(G\), $\gamma(G)\le 3\rho(G)$.
\end{restatable}

If Conjecture \ref{conj:UDG_conjecture}, the constant \(3\) would be tight, by the family constructed above.

We now turn our attention to subcubic graphs, that is, graphs $G$ with $\Delta(G) \leq  3$.
Henning et al. conjectured that if $\Delta(G) \leq 3$, then $\gamma(G) \leq 2 \rho(G)$, except for exactly three exceptions.
 This conjecture was disproved by Bonamy et. al \cite{bonamy2025}, who exhibited an infinite family of graphs $G$ with $\Delta(G) =3$ but $\gamma(G) =2 \rho(G) +1$.
However, they showed the next.

\begin{theorem}[{\cite{Henning2011}}]
\label{th:cubicclawfree-henning}
    If $G$ is claw-free and $\Delta(G) \leq 3$, then $\gamma(G) \leq 2 \rho(G)$.
\end{theorem}

More results on subclasses of graphs $G$ with
$\Delta(G) \leq 3$ were recently proven by
Henning et al. \cite{Henning2026}.
In this paper we focus on bridgeless claw-free cubic graphs, improving Theorem \ref{th:cubicclawfree-henning} as follows.

\begin{restatable}{theorem}{CubicClawFreeTheorem}\label{theorem:cubic-claw-free}
For every bridgeless claw-free cubic graph $G$, 
$
\gamma(G) \leq \frac{7}{4}\rho(G)+\frac{5}{6}.
$
\end{restatable}


The bound in Theorem \ref{theorem:cubic-claw-free} is not tight.
However, we construct an infinite family of bridgeless cubic claw free graphs that attain a factor of $5/4$.

\begin{theorem}
    There exists an infinite family of cubic bridgeless claw-free graphs $G$ with ${\gamma(G)= \frac{5}{4} \rho(G)}$.
\end{theorem}

This motivates the following conjecture.

\begin{conjecture} \label{conj:cubic-75}
    For every bridgeless claw-free cubic graph $G$, 
$
\gamma(G) \leq \frac{5}{4}\rho(G).
$
\end{conjecture}

Observe that Conjecture \ref{conj:cubic-75}, if true, would be best possible by the family above.

%% file: UDG.tex
A unit disk graph is the intersection graph of equal-radius disks in the plane. Equivalently, it is a graph $G=(V,E)$ for which there exists a representation $p:V\to\mathbb{R}^2$ such that $uv\in E$ if and only if $\|p(u)-p(v)\|_2\le 1$.

In this section, we improve the ratio for unit disk graphs, mentioned in \cite{DBLP:journals/corr/abs-2503-05562}. 

\begin{theorem}[\cite{DBLP:journals/corr/abs-2503-05562}\label{th:UDG_1}]
    Given a unit disk graph $G$, $\gamma(G)\leq 32 \cdot\rho(G)$.
\end{theorem}

We improve this ratio of $32$ to $16$. Throughout this section, we consider $G=(V,E)$ to be a unit disk graph, where every vertex $v \in V$ is represented by a point $p(v)=(x(v),y(v)) \in \mathbb{R}^2$.
For an integer $i$, define the horizontal strip $S_i$ as follows:

\[S_i = \{(x,y) \in \mathbb{R}^2 : \frac{i}{\sqrt{2}} \leq y < \frac{i+1}{\sqrt{2}}\}\]

We obtain a \emph{strip decomposition} $\mathcal{S}_G$ of $G$ by partitioning the vertex set into $V_i = \{v \in V : p(v) \in S_i\}$ for all $i \in \mathbb{Z}$. Observe that if $u \in V_i$ and $v \in V_j$ are adjacent, then $|i-j| \leq 2$. Indeed, if $|i-j| \geq 3$, then the vertical distance between $p(u)$ and $p(v)$ is at least $\sqrt{2}~>1$, which implies that $u$ and $v$ are not adjacent. 
The following observations will be useful.

\begin{observation}\label{obs:UDG_1}
    Given a unit disk graph $G=(V,E)$, let $u,v\in V(G)$ satisfy $\|p(u)-p(v)\|_2>2$.
Then $N[u]\cap N[v]=\emptyset$.
\end{observation}

\begin{observation}\label{obs:UDG_2}
    Given a unit disk graph $G=(V,E)$ and its strip decomposition 
   $ \mathcal S_G=\{V_i : i\in\mathbb Z\}$,
    let $u,v$ be two vertices such that $u\in V_i$ (resp. $v\in V_j$) and $|i-j|\geq 4$. Then $N[u]\cap N[v]=\emptyset$. 
\end{observation}

\begin{proof}
    Note that  $\|p(u)-p(v)\|_2\geq \frac{3\sqrt{2}}{2}>2$, since $|i-j|\geq 4$. By Observation \ref{obs:UDG_1}, $N[u]\cap N[v]=\emptyset$.
\end{proof}

Next, within each strip $S_i$, we further partition it into axis-aligned squares of side length $1/\sqrt{2}$. Formally, for integers $i,j \in \mathbb{Z}$, we define

\[
Q_{i,j} = \Big\{ (x,y) \in S_i : \tfrac{j}{\sqrt{2}} \leq x <  \tfrac{j+1}{\sqrt{2}} \Big\}.
\]
This produces a tiling of $\mathbb{R}^2$ into squares of side length $1/\sqrt{2}$.
Each vertex $v\in V$ is assigned to the unique square $Q_{i,j}$ containing $p(v)$.
We call the resulting partition of $V(G)$ the \emph{square decomposition} of $G$.
For each square $Q_{i,j}$, let
\[
V_{i,j}=\{v\in V : p(v)\in Q_{i,j}\}.
\]
In what follows, we consider only non-empty squares, that is, those $Q_{i,j}$ with $V_{i,j}\neq\emptyset$. For every non-empty square $Q_{i,j}$, fix a representative vertex $v_{i,j}\in V_{i,j}$ (if $|V_{i,j}|>1$ the choice may be arbitrary but fixed).
For each square $Q_{i,j}$, let $G_{i,j}=G[V_{i,j}]$ denote the induced subgraph of $G$ on the vertices contained in $Q_{i,j}$.
Since the diagonal of $Q_{i,j}$ has length $1$, any two points in $Q_{i,j}$ are at Euclidean distance at most $1$. Therefore, $G_{i,j}$ is a clique.

We now assign labels to the squares $Q_{i,j}$ using the function $\ell$
as follows:

\[
\ell(Q_{i,j}) \;=\; 4\big(i \bmod 4\big) \;+\; \big(j \bmod 4\big) \;+\; 1,
\]

Thus each strip $S_i$ uses four consecutive labels periodically to label its squares:

\[
\begin{aligned}
S_{4t} &\mapsto \{1,2,3,4\},\\
S_{4t+1} &\mapsto \{5,6,7,8\},\\
S_{4t+2} &\mapsto \{9,10,11,12\},\\
S_{4t+3} &\mapsto \{13,14,15,16\},
\end{aligned}
\]

and the pattern repeats with period $4$ in the vertical direction and period $4$ in the horizontal direction (see Figure \ref{fig:UDG_2}).

\begin{figure}[h!]
    \centering
    \includegraphics[width=0.7\linewidth]{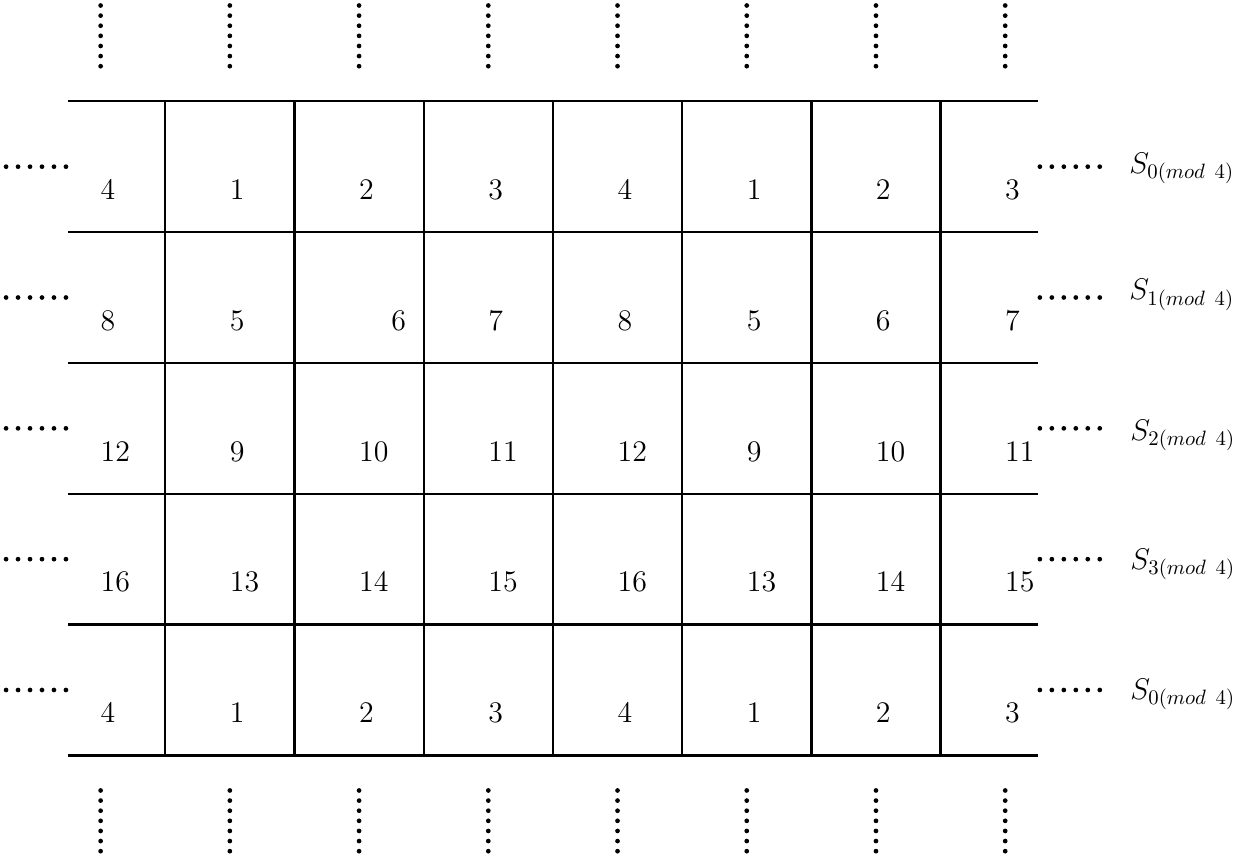}
    \caption{A square decomposition together with the corresponding labels.}
    \label{fig:UDG_2}
\end{figure}

Next we define the sets $O_k$ as follows:
$O_k:=\{v_{i,j} : \ell(Q_{i,j})=k\},\qquad k\in[16].$
We also define a set $D:=\bigcup_{k\in [16]} O_k$. 
The following observation follows from the geometric properties of the decomposition.

\begin{observation}\label{obs:UDG_3}
Given a unit disk graph $G$ and its square decomposition $\{Q_{i,j}\}$, the set $D$ is a dominating set of $G$, and for each $i\in [16]$, the set $O_i$ is a packing set of $G$.
\end{observation}
\begin{proof}
Since every non-empty square induces a clique, every vertex of $G$ is adjacent to the representative vertex chosen in its square. Hence $D$ is a dominating set.

Now let $v_{i,j},v_{i',j'}\in O_k$ be distinct. Since $\ell(Q_{i,j})=\ell(Q_{i',j'})$, we have $i\equiv i' \pmod 4$ and $j\equiv j' \pmod 4$. Thus either $|i-i'|\ge 4$ or $|j-j'|\ge 4$.
In the first case, Observation~\ref{obs:UDG_2} implies that $N[v_{i,j}]\cap N[v_{i',j'}]=\emptyset$. In the second case, the horizontal distance between $Q_{i,j}$ and $Q_{i',j'}$ is at least $3/\sqrt{2}>2$, and therefore Observation~\ref{obs:UDG_1} again implies that $N[v_{i,j}]\cap N[v_{i',j'}]=\emptyset$.
Hence $O_k$ is a packing.
\end{proof}

We now state and prove the main theorem of this section.

\UnitDiskTheorem*


\begin{proof}
Let $\{Q_{i,j}\}$ be the square decomposition of $G$ as defined above. By Observation \ref{obs:UDG_3}, the set $D$ is a dominating set of $G$, and each $O_k$ is a packing set of $G$.
Thus, $\gamma(G)\le |D|$ and $|O_k|\le \rho(G)$ for every $k\in[16]$.
Moreover, by construction,
$
|D|=\sum_{k\in[16]} |O_k|.
$
Therefore,
\[
\gamma(G)\le |D|
=\sum_{k\in[16]} |O_k|
\le 16\rho(G),
\]
as desired.
\end{proof}



\UnitDiskFamily*
\begin{proof}
    
We construct a unit disk graph \(H\) such that $\gamma(H)=3$ and 
$\rho(H)=1$. 

\begin{figure}[htbp]

\centering
\begin{tikzpicture}[
    scale=2.6,
    every node/.style={circle, draw, fill=white, inner sep=1.4pt, font=\scriptsize},
    outer/.style={circle, draw, fill=blue!12, inner sep=1.6pt, font=\scriptsize},
    middle/.style={circle, draw, fill=red!12, inner sep=1.4pt, font=\scriptsize},
    edge/.style={line width=0.45pt},
    medge/.style={line width=0.35pt, gray!70}
]

\coordinate (a)  at (0,0);
\coordinate (ap) at (0.60,-0.22);
\coordinate (b)  at (2,0);
\coordinate (bp) at (1.90,0.62);
\coordinate (c)  at (1,1.732);
\coordinate (cp) at (0.51,1.31);

\coordinate (x0) at (1,0);
\coordinate (x1) at (0.95,0.31);
\coordinate (y0) at (0.50,0.866);
\coordinate (y1) at (0.80,0.756);
\coordinate (z0) at (1.50,0.866);
\coordinate (z1) at (1.255,0.655);

\draw[medge] (x0)--(x1);
\draw[medge] (x0)--(y0);
\draw[medge] (x0)--(y1);
\draw[medge] (x0)--(z0);
\draw[medge] (x0)--(z1);

\draw[medge] (x1)--(y0);
\draw[medge] (x1)--(y1);
\draw[medge] (x1)--(z0);
\draw[medge] (x1)--(z1);

\draw[medge] (y0)--(y1);
\draw[medge] (y0)--(z0);
\draw[medge] (y0)--(z1);

\draw[medge] (y1)--(z0);
\draw[medge] (y1)--(z1);

\draw[medge] (z0)--(z1);

\draw[edge] (a)--(ap);
\draw[edge] (b)--(bp);
\draw[edge] (c)--(cp);

\draw[edge] (a)--(x0);
\draw[edge] (a)--(x1);
\draw[edge] (a)--(y0);

\draw[edge] (ap)--(x0);
\draw[edge] (ap)--(x1);
\draw[edge] (ap)--(y1);

\draw[edge] (b)--(x0);
\draw[edge] (b)--(z0);
\draw[edge] (b)--(z1);

\draw[edge] (bp)--(x1);
\draw[edge] (bp)--(z0);
\draw[edge] (bp)--(z1);

\draw[edge] (c)--(y0);
\draw[edge] (c)--(y1);
\draw[edge] (c)--(z0);

\draw[edge] (cp)--(y0);
\draw[edge] (cp)--(y1);
\draw[edge] (cp)--(z1);

\node[outer, label=below left:{$a$}] at (a) {};
\node[outer, label=below:{$a'$}] at (ap) {};
\node[outer, label=below right:{$b$}] at (b) {};
\node[outer, label=right:{$b'$}] at (bp) {};
\node[outer, label=above:{$c$}] at (c) {};
\node[outer, label=left:{$c'$}] at (cp) {};

\node[middle, label=below:{$x_0$}] at (x0) {};
\node[middle, label=below right:{$x_1$}] at (x1) {};
\node[middle, label=left:{$y_0$}] at (y0) {};
\node[middle, label=above right:{$y_1$}] at (y1) {};
\node[middle, label=right:{$z_0$}] at (z0) {};
\node[middle, label=below right:{$z_1$}] at (z1) {};
\end{tikzpicture}
\caption{The unit disk graph $H$}
\label{fig:UDG_ex}

\end{figure}
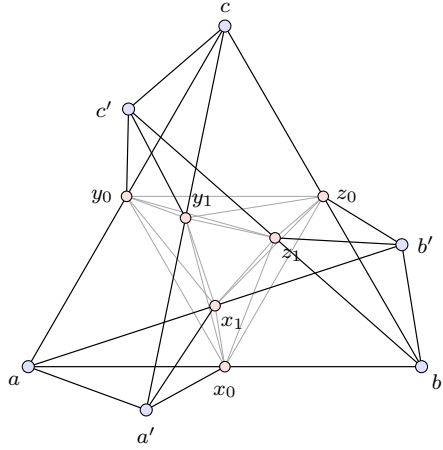

The graph is $H=(V,E)$ with vertex set
$V=\{a,b,c,a',b',c',x_0,x_1,y_0,y_1,z_0,z_1\}$.
The coordinates are given by:
\begin{itemize}
    \item $a=(0,0),~b=(2,0),~c=(1,\sqrt{3})$,
    \item $x_0=\frac{a+b}{2},~y_0=\frac{a+c}{2},~z_0=\frac{b+c}{2}$,
    \item $a'=(0.6,-0.22),~b'=(1.9,0.62),~c'=(0.51,1.31)$,
    \item $x_1=\frac{a+b'}{2},~y_1=\frac{a'+c}{2},~z_1=\frac{b+c'}{2}$.
\end{itemize}
See Figure~\ref{fig:UDG_ex} for a clear understanding of the graph.
Since $H$ has diameter $2$, every pair of vertices is at distance at most $2$, and therefore $\rho(H)=1$.
As $\{a,b,c\}$ forms an optimal dominating set of $H$, we have $\gamma(H)=3$. 
Hence $\gamma(H)=3\rho(H)$.

We can strengthen the construction to obtain an infinite family of connected unit disk graphs with the same ratio. The key observation is that the class of unit disk graphs is closed under true-twin blow-ups \cite{DBLP:journals/dm/ClarkCJ90}. More precisely, given a unit disk representation of a graph $G$, any vertex $u$ can be replaced by a clique $C_u$ of arbitrary size by placing the vertices of $C_u$ sufficiently close to the point representing $u$. Then every vertex of $C_u$ has the same neighborhood outside $C_u$ as $u$, while the vertices of $C_u$ are pairwise adjacent. Consequently, the resulting graph is again a unit disk graph.


Let $H$ be the unit disk graph constructed above, satisfying
$\gamma(H)=3$ and $\rho(H)=1$.
For each integer $t\ge 1$, let $H_t$ be obtained from $H$ by replacing a fixed vertex $u\in V(H)$ by a clique $C_u$ of size $t$, where every vertex of $C_u$ has the same neighbors outside $C_u$ as $u$ had in $H$. Since $H$ is connected, each $H_t$ is connected.
Since the vertices of $C_u$ are pairwise true twins, any dominating set of $H$ can be viewed as a dominating set of $H_t$, and vice versa. Hence $\gamma(H_t)=\gamma(H)=3$.
Similarly, every packing of $H_t$ contains at most one vertex of $C_u$. Replacing a vertex of $C_u$ by $u$ yields a packing of $H$, and conversely every packing of $H$ is also a packing of $H_t$. Therefore $\rho(H_t)=\rho(H)=1$.

Consequently,
\[
\frac{\gamma(H_t)}{\rho(H_t)}=3.
\]
Since $t$ is arbitrary, this yields an infinite family of connected unit disk graphs attaining the ratio $3$.
\end{proof}
It is natural to ask whether this example is extremal. We attempted to construct unit disk graph $G$ with
$
\frac{\gamma(G)}{\rho(G)}>3,
$
by considering several different geometric configurations. None of these attempts produced a unit disk graph with ratio strictly larger than \(3\). In particular, the main obstruction is that forcing the packing number to remain small tends to create enough short-distance connections to make the graph efficiently dominated. This motivates the following conjecture:


\UnitDiskConjecture

If true, the constant \(3\) would be tight, by the family constructed above.

%% file: claw-free-cubic.tex
Henning \emph{et al.} showed that $\gamma(G) \le 2\rho(G)$ for every claw-free subcubic graph $G$ (with a finite number of exceptions). In this section, we improve this bound for graphs that are claw-free, cubic, and bridgeless.

We consider the following greedy algorithm, first introduced by Favaron to show that $\rho(G) \ge (n-2)/8$ for any connected cubic graph~\cite{Favaron1999}.

\medskip
\noindent\textbf{Algorithm 1.} \emph{Initialize $P := \{x\}$ for an arbitrary vertex $x \in V(G)$; while there exists $u \in N_3(P)$, set $P := P \cup \{u\}$.}
\medskip

We next show that this bound remains valid for subcubic multigraphs.

\begin{lemma} \label{lemma:subcubic-leq-n8}
For any subcubic multigraph $G$ of order $n$, we have $\rho(G) \ge (n-2)/8$.
\end{lemma}

\begin{proof}
Apply Algorithm~1 and let $P$ be the resulting set, and $Q = N_{\le 2}(P)$. We show that $|P| \ge |Q|/8 - 1/4$ throughout.

\begin{claim}\label{claim:cubic-Nleq2leq10}
For every vertex $u \in V(G)$, $|N_{\le 2}(u)| \le 10$.
\end{claim}

\begin{proof}
Since $G$ is subcubic, $u$ has at most three neighbors, each contributing at most two new vertices at distance two, so $|N_{\le 2}(u)| \le 1+3+6=10$.
\end{proof}

Initially, $P=\{x\}$ and thus $|Q|=|N_{\le 2}(x)| \le 10$ by Claim~\ref{claim:cubic-Nleq2leq10}, so $|P|=1 \ge |Q|/8 - 1/4$. Now suppose a vertex $u$ is added to $P$, and let $q$ be the number of new vertices added to $Q$. Since $u \in N_3(P)$, there exists a path $pu''u'u$ with $p \in P$, $u'' \in N(P)$ and $u' \in N_2(P)$, so $u',u'' \in Q$ before adding $u$. Hence $q \le 8$, and the inequality is preserved.
At termination $Q=V(G)$, so $|P| \ge n/8 - 1/4 \ge (n-2)/8$.
\end{proof}

Lemma \ref{lemma:subcubic-leq-n8} can be extended to claw-free cubic graphs, by showing that $\rho(G) \geq \frac{n-2}{6}$. However, to improve the factor of $2$ given by Henning, we require a refined bound, as follows.
For a graph $G$, let $S(G)$ denote the set of vertices of $G$ that belongs to a subgraph of $G$ isomorphic to the house graph (Figure \ref{fig:claw-free-cubic}a).

\begin{lemma} \label{lemma:claw-free-rho-geqfrac{n-2}{6} +fracS(G)30}
    For any connected claw-free cubic graph $G$,
    on $n$ vertices, $\rho(G) \geq \frac{n-2}{6} + \frac{|S(G)|}{30}$.
\end{lemma}
\begin{proof}
Set $S:=S(G)$ and $\overline{S}= V(G) \setminus S$.
   Apply Algorithm 1.
    We will show that, after each addition, $$|P| \geq \frac{|Q|}{6}+\frac{|R|}{5}-\frac{1}{3},$$ where 
    $Q=N_{\leq 2}(P) \cap \overline{S}$ and
    $R= N_{\leq 2}(P) \cap S$.
 Hence, at the end of the algorithm, as $N_{\leq 2}(P) =V(G)$, we have $|P| \geq \frac{|\overline{S}|+|S|}{6}+\frac{|S|}{30}-\frac{1}{3}=\frac{n-2}{6}+\frac{|S|}{30}.$
    \input{Figures/fig-claw-free-cubic.tex}

 \begin{claim}\label{claim:cubicclawfree-Nleq2leq8}
     For every vertex $u \in V(G)$, $|N_{\leq 2}(u)| \leq 8$. Moreover, if $|N_{\leq 2}(u)| = 8$, then $N_{\leq 2}(u)$ has the graph of Figure \ref{fig:claw-free-cubic}b as subgraph.
 \end{claim}
 \begin{proof}
    Let $N(u)=\{a,b,c\}$. As $G$ is claw-free, we may assume without loss of generality, that $ab \in E(G)$ so each of $a$ and $b$ has at most one neighbor in $N_2(u)$. As $c$ has at most 2 neighbors in $N_2(u)$, we must have $N_2(u) \leq 4$. Hence $|N_{\leq 2}(u)| \leq 8$. Now, equality holds when $a$ and $b$ has exactly one neighbor, different from each other in $N_2(u)$ and $c$ has exactly 2 such neighbors, which should be adjacent, as $G$ is . Thus $N_{\leq 2}(u)$ has the graph of Figure \ref{fig:claw-free-cubic}b as a subgraph.
 \end{proof}

 We begin by showing that this is the case after the first iteration.
     For this, by Claim \ref{claim:cubicclawfree-Nleq2leq8}, $|N_{\leq 2}(x)| \leq 8$.
   As $|P|=1$ and $|Q|+|R|= |N_{\leq 2}(x)|$, the statement holds.

   Suppose now that $u$ was added to $P$ in an arbitrary iteration of the algorithm. As $u \in N_3(P)$, $u$ has a neighbor $u' \in N_2(P)$, which in turn has a neighbor $u'' \in N(P)$ for some $p \in P$.
   Let $x$ be the neighbor of $u'$ distinct from $u$ and $u''$ (so $N(u')=\{u,u'',x\}$).
   Note that, as $|P-u|=|P|-1$, it suffices to show that $\frac{q}{6}+ \frac{r}{5} \leq 1$, where $q$ and $r$ are the number of vertices added to $Q$
   and $R$, respectively.
   If
   $q+r \leq 5$ then the proof follows.
   Hence, assume that $q+r \geq 6$, which implies that 
   $|N_{\leq2}(u)|=8$
   and $x \in N_3(P)$. As $G$ is claw-free, we must have $ux \in E(G)$.

    We will show that $q \geq 1$. Suppose by contradiction that $q=0$ (and $r=6)$, so all new vertices added to $N_{\leq 2}(P)$  belongs to $S$.
    As $u$ was recently added to this set, there 
    is a subgraph isomorphic to the house graph that contains $u$. So, $u$ belongs to a triangle, say $u\alpha \beta$.

    By Claim \ref{claim:cubicclawfree-Nleq2leq8}, $N_{\leq 2}(u)$ has the graph of Figure \ref{fig:claw-free-cubic}b as a subgraph.
    Considering the labels in such figure,
    we must have $a= \alpha$ and $b= \beta$.
    Now, as $ca',cb' \notin E(G)$, we must have $a'b' \in E(G)$.
    Suppose for a moment that $c=u'$. So, without loss of generality, $c'=u''$.
    This implies that $c'' \in N_{\leq 2}(P-u)$, a contradiction to the fact that $r=6$. Hence, we may assume, without loss of generality, that $u'=a$.
    If $u''=b$ then $u \notin N_3(P-u)$, hence $u''=a'$.
    But, as $a'b' \in E(G)$, we have $b' \in N_{\leq 2}(P-u)$, again a contradiction.
\end{proof}

We use the following characterization due to Oum~\cite{Oum2011}.

\begin{lemma}[Proposition~1 of~\cite{Oum2011}]\label{lemma:claw-free-characterization}
A graph $G$ is cubic, bridgeless, and claw-free if and only if exactly one of the following holds:
\begin{itemize}
    \item $G \cong K_4$,
    \item $G$ is a ring of diamonds, or
    \item there exists a cubic multigraph $H$ such that $G$ is obtained from $H$ by replacing each edge with a (possibly trivial) string of diamonds and each vertex with a triangle.
\end{itemize}
\end{lemma}

In this context, a \emph{diamond} is an induced subgraph isomorphic to $K_4 - e$. 
We say that two diamonds are \emph{adjacent} if there exists an edge joining a degree-$2$ vertex of one diamond to a degree-$2$ vertex of the other. 
Contracting each diamond to a single vertex and preserving these adjacencies yields an auxiliary graph; if this graph is a path, we call the corresponding subgraph a \emph{string of diamonds}, and if it is a cycle, we call it a \emph{ring of diamonds}.
(Figure~\ref{fig:cubic-claw-free-ring}).

\input{Figures/cubic-claw-free-ring}

\CubicClawFreeTheorem*
\begin{proof}
If $G \cong K_4$ then the proof is clear. 
Otherwise, if $G$ is a ring of diamonds, we construct a dominating set that is also a packing by selecting an arbitrary vertex in each diamond.
Hence, by Lemma \ref{lemma:claw-free-characterization}, $G$ is obtained from a cubic multigraph, say $H$.
    Let $G'$ be the graph obtained from $H$ by replacing each vertex of $H$ with a triangle (that is, before replacing edges with strings of diamonds).
We first construct a solution in $G'$ and then extend it to $G$. Set $n= |V(G')|$, so $|V(H)|=n/3$. 

Let $S$ be the set of vertices of $H$ that belong to a $2$-cycle, and let $Q$ be a maximum packing in $H-S$ (Figure \ref{fig:cubic-claw-free-H}a).
For any $u \in V(H)$, we let $t(u)$ be the corresponding triangle in $G$
and define $t(X)$ accordingly for any $X \subseteq V(H)$.
Also, for any $u \in V(G)$, we denote by $h(u)$ the vertex in $H$ such that $u \in t(h(u))$.

\input{Figures/cubic-claw-free-H}

\begin{claim}\label{claim:gammaG'leqfracn3-Q}
    $\gamma(G') = \frac{n}{3}-\rho(H)$.
\end{claim}
\begin{proof}
    Let $Q$ be a maximum packing in $H$.
    Consider a dominating set $D'$ in $G'$ formed as follows. Let $a \in V(Q)$ and $t(a)=\{a_1,a_2,a_3\}$. For every $a_i$, let $a'_i$ be the (unique) neighbor of $a_i$ not in $t(a)$.
    Then, we add $\{a'_1,a'_2,a'_3\}$ to $D'$ for any such $a$. After that, for every vertex $u\in N_{\leq 2} (Q)$, we pick an arbitrary vertex in $t(u)$ and add it to $D'$ (Figure \ref{fig:cubic-claw-free-H}b).

    It is clear that $D'$ dominates $G'$.
    Indeed, for each $u \notin Q$, the triangle $t(u)$ contains a selected vertex. 
For each $a \in Q$, every vertex of $t(a)$ is adjacent to one of $a'_1,a'_2,a'_3 \in D'$.
    Now, as $Q$ is a packing in $H$, no two vertices in the same triangle of $G'$ are selected. Thus, we selected exactly one vertex from each triangle of $G'$, except from the triangles that corresponds to some vertex in $Q$. Hence $|D'|=\frac{n}{3}-|Q|$.
    
    Conversely, let $D'$ a minimum dominating set in $G'$. Moreover, let us assume that $D'$ minimizes the number of vertices $u \in V(H)$ such that $|t(u) \cap D'| \geq 2$.
    We first show that there exists no such vertices. Indeed, suppose for a contradiction that there exists a vertex
     $u \in V(H)$ such that $|t(u) \cap D'| \geq 2$. Suppose $t(u) = \{u_1,u_2,u_3\}$
     and that $u_1,u_2 \in D'$.
     For every $u_i$, let $u'_i$ be the unique neighbor of $u_i$ not in $t(u)$.
     Note that, as $D'$ is minimal, no other neighbor of $u'_1$ is in $D'$, but then $(D' \setminus \{u_1\}) \cup \{u'_1\}$ is also a minimum dominating set, contradicting the choice of $D'$
     
    Consider the set $Q \subseteq V(H)$ formed by the vertices $u \in V(H)$ with $t(u) \cap D' = \emptyset$.
    We will show that $Q$ is a packing.
    Let $u,v \in Q$ and suppose for a contradiction that $u$ and $v$ are at distance at most 2.
    Let $t(u)= \{u_a,u_b,u_c\}, t(v)= \{v_a,v_b,v_c\}$
    If $uv \in E(H)$, then we
     may assume without of generality that $u_av_a \in E(G')$. But in that case, both $u_a$ and $v_a$ are not dominated by $D'$, a contradiction.
     Hence, assume that there exists a vertex $x \in V(H)$ such that $x \in N(u) \cap N(v)$. Let $t(x)=\{x_a,x_b,x_c\}$
     and assume without loss of generality that
     $u_ax_a,v_bx_b \in V(G')$.
     In this case, by the choose of $D'$ at least one of $\{x_a,x_b\}$ is not in $D'$
     and hence one of $\{u_a,v_b\}$ is not dominated by $D'$, a contradiction.
\end{proof}

  Let $Q$ be a maximum packing in $H$.
A vertex $v \in V(G')$ belongs to $t(s)$ for some $s \in S$ if and only if $v \in S(G')$. Hence,
by Lemma~\ref{lemma:claw-free-rho-geqfrac{n-2}{6} +fracS(G)30}, $\rho(G') \geq \frac{n-2}{6}+\frac{|t(S)|}{30}=\frac{n-2}{6}+\frac{|S|}{10}$.
Note that $H-S$ is a subcubic multigraph, hence, by Lemma~\ref{lemma:subcubic-leq-n8}, 
\[
|Q|=\rho(H-S) \geq \frac{n/3-|S|-2}{8}.
\]
This, together with Claim~\ref{claim:gammaG'leqfracn3-Q}, implies that
\[
\gamma(G') \leq \frac{n}{3} - \left(\frac{n}{24}-\frac{|S|}{8}-\frac{1}{4}\right)
= \frac{7n}{24} + \frac{|S|}{8} + \frac{1}{4}.
\]
Using again that $\rho(G') \geq \frac{n-2}{6}+\frac{|S|}{10}$, we obtain
\[
n \leq 6\rho(G') + 2 - \frac{3|S|}{5}.
\]
Substituting this into the previous inequality yields
\[
\gamma(G') \leq \frac{7}{4}\rho(G') - \frac{1}{20}|S| + \frac{5}{6}
\leq \frac{7}{4}\rho(G') + \frac{5}{6}.
\]

Finally, we argue by induction on the number of edges of $H$ that are replaced by strings of diamonds. 
If no edge is replaced, then $G = G'$ and the result follows from the previous paragraph. 
Otherwise, let $uv$ be an edge of $G'$ that is replaced by a string of diamonds. 
By the inductive hypothesis, there exists a dominating set $D'$ and a packing $P'$ in the graph before replacing $uv$ satisfying $|D'| \le \frac{7}{4}|P'| + \frac{5}{6}$.
Suppose first that $u,v \notin D'$.
As $D'$ is a dominating set, both $u$ and $v$ are still dominated by vertices of $G'$.
Moreover, we can select one vertex for each diamond added to form a dominating set for the new added vertices.

\input{Figures/cubic-claw-free-string}

Now, note that at most one of $\{u,v\}$ is in $P'$.
If none of them are in $P'$, we pick an arbitrary vertex in each diamond of the string of diamonds and add it to $P'$ (Figure  \ref{fig:cubic-claw-free-string}a).
Otherwise, we may assume without loss of generality that $u \notin P'$ and $v \in P'$.
In this case, we can also select a vertex in each diamond of the string of diamonds and add it to $P'$,
however, we must start by selecting a vertex in a diamond that is a neighbor of $u$ in $G$ (Figure \ref{fig:cubic-claw-free-string}b). As $v \in P'$, no vertex in $P' \setminus \{v\}$ is a neighbor of $u$ in $G'$, so this set is a packing in $G$.
\end{proof}

%% file: Figures/fig-claw-free-cubic.tex
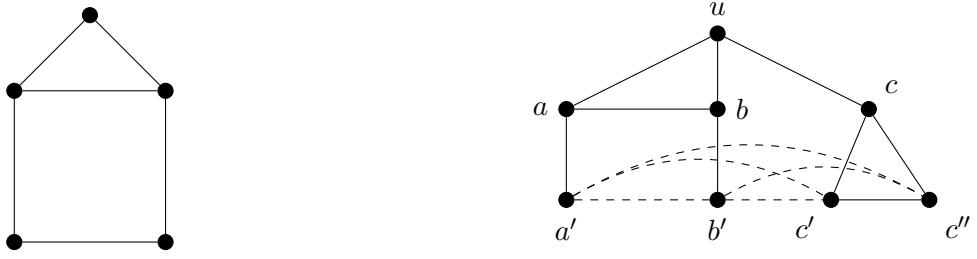
\begin{figure}[ht]
\centering

\begin{subfigure}{0.45\textwidth}
\centering
\begin{tikzpicture}[
    vertex/.style={circle, draw, fill=black, inner sep=2pt},
    edge/.style={draw}
]
\node[vertex] (v1) at (0,0) {};
\node[vertex] (v2) at (2,0) {};
\node[vertex] (v3) at (2,2) {};
\node[vertex] (v4) at (0,2) {};
\node[vertex] (v5) at (1,3) {};

\draw (v1)--(v2)--(v3)--(v4)--(v1);
\draw (v3)--(v5)--(v4);
\end{tikzpicture}

\end{subfigure}
\hfill
\begin{subfigure}{0.45\textwidth}
\centering

\begin{tikzpicture}[
    vertex/.style={circle, draw, fill=black, inner sep=2pt},
    edge/.style={draw}
]

\node[vertex,label=above:$u$] (u) at (0,3) {};

\node[vertex,label=left:$a$] (a) at (-2,2) {};
\node[vertex,label=right:$b$] (b) at (0,2) {};
\node[vertex,label=above right:$c$] (c) at (2,2) {};

\node[vertex,label=below:$a'$] (a1) at (-2,0.8) {};
\node[vertex,label=below:$b'$] (b1) at (0,0.8) {};
\node[vertex,label=below left:$c'$] (c1) at (1.5,0.8) {};
\node[vertex,label=below right:$c''$] (c2) at (2.8,0.8) {};

\draw (u) -- (a);
\draw (u) -- (b);
\draw (u) -- (c);

\draw (a) -- (b);
\draw (a) -- (a1);
\draw (b) -- (b1);

\draw (c) -- (c1);
\draw (c) -- (c2);
\draw (c1) -- (c2);

\draw[dashed] (a1) -- (b1);
\draw[dashed, bend right=-30] (a1) to (c1);
\draw[dashed, bend right=-30] (a1) to (c2);
\draw[dashed] (b1) -- (c1);
\draw[dashed, bend right=-30] (b1) to (c2);

\end{tikzpicture}

\end{subfigure}

\caption{$(a)$ The house graph. $(b)$ The subgraph mentioned in the proof of Claim \ref{claim:cubicclawfree-Nleq2leq8}. Dashed edges may or not may exist in $N_{\leq2}(u)$.}
\label{fig:claw-free-cubic}
\end{figure}

%% file: Figures/cubic-claw-free-ring.tex
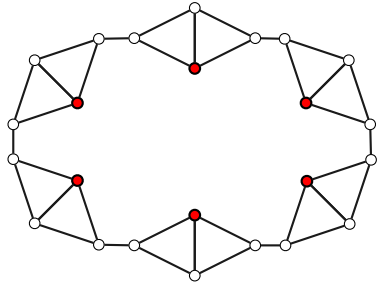
\begin{figure}[H]
\centering
\begin{tikzpicture}[scale=0.4, transform shape]
   
    \begin{scope}[yshift=1cm]
    \node[vertex] (A0) {};
   \node[vertex] (A1) at ($(A0)+(4,0)$) {};
    \node[vertex, fill=red, draw=black, thick] (A2) at ($(A0)+(2,1)$) {};
    \node[vertex] (A3) at ($(A0)+(2,-1)$) {};
    \foreach \source/\dest in {A0/A2,A0/A3,A1/A2,A1/A3,A2/A3}
        \path[black edge] (\source) -- (\dest);
    \end{scope}

    \draw[black edge]
      (A2) -- (A3);
        
    \begin{scope}[xshift=5cm,yshift=1cm, rotate=45]
    \node[vertex] (B0) {};
    \node[vertex] (B1) at ($(B0)+(4,0)$) {};
    \node[vertex, fill=red, draw=black, thick] (B2) at ($(B0)+(2,1)$) {};
    \node[vertex] (B3) at ($(B0)+(2,-1)$) {};
    \foreach \source/\dest in {B0/B2,B0/B3,B1/B2,B1/B3,B2/B3}
        \path[black edge] (\source) -- (\dest);
    \end{scope}

      \draw[black edge]
      (B2) -- (B3);

    \begin{scope}[xshift=7.8cm,yshift=5cm, rotate=135]
    \node[vertex] (C0) {};
    \node[vertex] (C1) at ($(C0)+(4,0)$) {};
    \node[vertex, fill=red, draw=black, thick] (C2) at ($(C0)+(2,1)$) {};
    \node[vertex] (C3) at ($(C0)+(2,-1)$) {};
    \foreach \source/\dest in {C0/C2,C0/C3,C1/C2,C1/C3,C2/C3}
        \path[black edge] (\source) -- (\dest);
    \end{scope}

     \draw[black edge]
      (C2) -- (C3);

    \begin{scope}[yshift=7.85cm]
    \node[vertex] (D0) {};
    \node[vertex] (D1) at ($(D0)+(4,0)$) {};
    \node[vertex] (D2) at ($(D0)+(2,1)$) {};
    \node[vertex, fill=red, draw=black, thick] (D3) at ($(D0)+(2,-1)$) {};
    \foreach \source/\dest in {D0/D2,D0/D3,D1/D2,D1/D3,D2/D3}
        \path[black edge] (\source) -- (\dest);
    \end{scope}

     \draw[black edge]
      (D2) -- (D3);

    \begin{scope}[xshift=-4cm,yshift=3.85cm, rotate=-45]
    \node[vertex] (F0) {};
    \node[vertex] (F1) at ($(F0)+(4,0)$) {};
    \node[vertex, fill=red, draw=black, thick] (F2) at ($(F0)+(2,1)$) {};
    \node[vertex] (F3) at ($(F0)+(2,-1)$) {};
    \foreach \source/\dest in {F0/F2,F0/F3,F1/F2,F1/F3,F2/F3}
        \path[black edge] (\source) -- (\dest);
    \end{scope}

    \draw[black edge]
      (F2) -- (F3);

    \begin{scope}[xshift=-4cm,yshift=5cm, rotate=45]
    \node[vertex] (E0) {};
    \node[vertex] (E1) at ($(E0)+(4,0)$) {};
    \node[vertex] (E2) at ($(E0)+(2,1)$) {};
    \node[vertex, fill=red, draw=black, thick] (E3) at ($(E0)+(2,-1)$) {};
    \foreach \source/\dest in {E0/E2,E0/E3,E1/E2,E1/E3,E2/E3}
        \path[black edge] (\source) -- (\dest);
    \end{scope}

        \draw[black edge]  (E2) -- (E3);

\draw[black edge]      (A0) -- (F1);
    \draw[black edge] (D0) -- (E1);
    \draw[black edge] (E0) -- (F0);
    \draw[black edge] (C1) -- (D1);
    \draw[black edge] (B1) -- (C0);
    \draw[black edge] (A1) -- (B0);

\end{tikzpicture}
\label{fig:claw-free-a}
\caption{(a) A ring of diamonds. Filled vertices form a dominating set which is also a packing.}
\label{fig:cubic-claw-free-ring}
\end{figure}

%% file: Figures/cubic-claw-free-H.tex
\begin{figure}[H]
\centering

\begin{subfigure}{0.48\textwidth}
\centering
\begin{tikzpicture}[scale=1, 
  every node/.style={circle, draw, fill=black, inner sep=2pt}]

  \node[label=below left:$a$, color=red] (a) at (0,0) {};
  \node[label=below right:$c$] (c) at (2,0) {};
  \node[label=above right:$d$] (d) at (2,2) {};
  \node[label=above left:$b$] (b) at (0,2) {};

  \node[label=above:$e$] (e) at ($(b)!1/3!(d)$) {};
  \node[label=above:$f$] (f) at ($(b)!2/3!(d)$) {};

  \node[label=below:$g$] (g) at ($(a)!1/3!(c)$) {};
  \node[label=below:$h$] (h) at ($(a)!2/3!(c)$) {};

  \draw (a) -- (g);
  \draw (h) -- (c);
  \draw (c) -- (d);
  \draw (d) -- (f);
  \draw (e) -- (b);
  \draw (b) -- (a);

  \draw (a) -- (d);
  \draw (c) -- (b);

  \draw[bend left=30] (e) to (f);
  \draw[bend right=30] (e) to (f);

  \draw[bend left=30] (g) to (h);
  \draw[bend right=30] (g) to (h);

\end{tikzpicture}
\caption{}
\label{fig:claw-free-a}
\end{subfigure}
\hfill
\begin{subfigure}{0.48\textwidth}
\centering
\begin{tikzpicture}[scale=0.5, transform shape]

\tikzset{
  vertex/.style={circle, draw, fill=black, inner sep=2pt}
}

    \begin{scope}[xshift=-6cm, yshift=1cm, rotate=60]
    \node[vertex,label=above:$a_0$] (A0) {};
    \node[vertex,label=right:$a_1$] (A1) at ($(A0)+(2,0)$) {};
    \node[vertex,label=above:$a_2$] (A2) at ($(A0)+(1,{1*sqrt(3)})$) {};
    \foreach \source/\dest in {A0/A1,A1/A2,A2/A0}
        \path[black edge] (\source) -- (\dest);
    \end{scope}

    \begin{scope}[xshift=-1cm, yshift=1cm, rotate=60]
    \node[vertex,label=below:$g_0$, color=red] (G0) {};
    \node[vertex,label=left:$g_1$] (G1) at ($(G0)+(2,0)$) {};
    \node[vertex,label=above:$g_2$] (G2) at ($(G0)+(1,{1*sqrt(3)})$) {};
    \foreach \source/\dest in {G0/G1,G1/G2,G2/G0}
        \path[black edge] (\source) -- (\dest);
    \end{scope}

    \begin{scope}[xshift=2cm, yshift=1cm, rotate=60]
    \node[vertex,label=above:$h_0$] (H0) {};
    \node[vertex,label=below:$h_1$] (H1) at ($(H0)+(2,0)$) {};
    \node[vertex,label=right:$h_2$] (H2) at ($(H0)+(1,{1*sqrt(3)})$) {};
    \foreach \source/\dest in {H0/H1,H1/H2,H2/H0}
        \path[black edge] (\source) -- (\dest);
    \end{scope}

    \begin{scope}[xshift=6cm, yshift=1cm, rotate=60]
    \node[vertex,label=below:$c_0$] (C0) {};
    \node[vertex,label=below:$c_1$] (C1) at ($(C0)+(2,0)$) {};
    \node[vertex,label=above:$c_2$] (C2) at ($(C0)+(1,{1*sqrt(3)})$) {};
    \foreach \source/\dest in {C0/C1,C1/C2,C2/C0}
        \path[black edge] (\source) -- (\dest);
    \end{scope}

    \begin{scope}[xshift=-7cm, yshift=7cm]
    \node[vertex,label=left:$b_0$, color=red] (B0) {};
    \node[vertex,label=below:$b_1$] (B1) at ($(B0)+(2,0)$) {};
    \node[vertex,label=above:$b_2$] (B2) at ($(B0)+(1,{1*sqrt(3)})$) {};
    \foreach \source/\dest in {B0/B1,B1/B2,B2/B0}
        \path[black edge] (\source) -- (\dest);
    \end{scope}

    \begin{scope}[xshift=-2cm, yshift=7cm]
    \node[vertex,label=below:$e_0$] (E0) {};
    \node[vertex,label=above right:$e_1$] (E1) at ($(E0)+(2,0)$) {};
    \node[vertex,label=above:$e_2$] (E2) at ($(E0)+(1,{1*sqrt(3)})$) {};
    \foreach \source/\dest in {E0/E1,E1/E2,E2/E0}
        \path[black edge] (\source) -- (\dest);
    \end{scope}

    \begin{scope}[xshift=1cm, yshift=7cm]
    \node[vertex,label=above right:$f_0$] (F0) {};
    \node[vertex,label=below:$f_1$] (F1) at ($(F0)+(2,0)$) {};
    \node[vertex,label=above:$f_2$] (F2) at ($(F0)+(1,{1*sqrt(3)})$) {};
    \foreach \source/\dest in {F0/F1,F1/F2,F2/F0}
        \path[black edge] (\source) -- (\dest);
    \end{scope}

    \begin{scope}[xshift=5cm, yshift=7cm]
    \node[vertex,label=below:$d_0$, color=red] (D0) {};
    \node[vertex,label=below:$d_1$] (D1) at ($(D0)+(2,0)$) {};
    \node[vertex,label=above:$d_2$] (D2) at ($(D0)+(1,{1*sqrt(3)})$) {};
    \foreach \source/\dest in {D0/D1,D1/D2,D2/D0}
        \path[black edge] (\source) -- (\dest);
    \end{scope}

    \begin{scope}[xshift=-7cm, yshift={6cm}]
    \node[vertex,label=left:$w_0$] (W0) {};
    \node[vertex,label=right:$w_1$] (W1) at ($(W0)+(0,-2)$) {};
    \node[vertex,label=left:$w_2$] (W2) at ($(W0)+(-0.5,-1)$) {};
    \node[vertex,label=right:$w_3$] (W3) at ($(W0)+(0.5,-1)$) {};
    \foreach \source/\dest in {W0/W2,W0/W3,W1/W2,W1/W3,W2/W3}
        \path[black edge] (\source) -- (\dest);
    \end{scope}
    
    \begin{scope}[xshift=7cm, yshift=6cm]
    \node[vertex,label=left:$y_0$] (Y0) {};
    \node[vertex,label=right:$y_1$] (Y1) at ($(Y0)+(0,-2)$) {};
    \node[vertex,label=left:$y_2$] (Y2) at ($(Y0)+(-0.5,-1)$) {};
    \node[vertex,label=right:$y_3$] (Y3) at ($(Y0)+(0.5,-1)$) {};
    \foreach \source/\dest in {Y0/Y2,Y0/Y3,Y1/Y2,Y1/Y3,Y2/Y3}
        \path[black edge] (\source) -- (\dest);
    \end{scope}

    \draw[black edge] (A1) -- (D0);
    \draw[black edge] (B1) -- (C2);
    \draw[black edge] (A2) -- (W1);
    \draw[black edge] (B2) -- (E2);
    \draw[black edge] (D2) -- (F2);
    \draw[black edge] (Y0) -- (D1);
    \draw[black edge] (A0) -- (G0);
    \draw[black edge] (C0) -- (H0);
    \draw[black edge] (B0) -- (W0);
    \draw[black edge] (C1) -- (Y1);

    \draw[black edge, bend left=20] (G2) to (H2);
    \draw[black edge, bend left=20] (G1) to (H1);
    \draw[black edge, bend left=-20] (E0) to (F0);
    \draw[black edge, bend left=-20] (E1) to (F1);

\end{tikzpicture}
\caption{}
\label{fig:claw-free-b}
\end{subfigure}

\caption{(a) A cubic multigraph $H$. Here $S=\{e,f,g,h\}$ and $Q=\{a\}$ is a maximum packing of $H-S$. (b) A claw-free cubic graph obtained from $H$. In this graph, $D'$ selects vertices $b_0,d_0$ and $g_0$, after that, $D'$ selects four extra vertices, one vertex from each of the triangles $t(e),t(f),t(h)$ and $t(c)$.}
\label{fig:cubic-claw-free-H}
\end{figure}
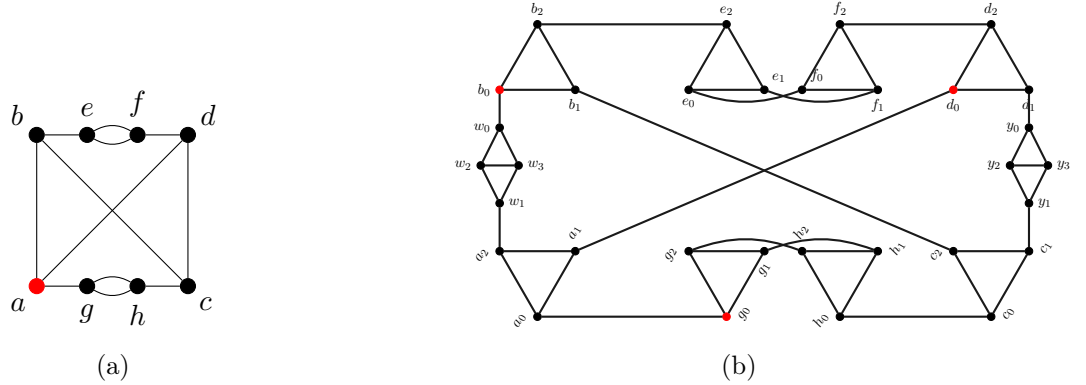

%% file: Figures/cubic-claw-free-string.tex
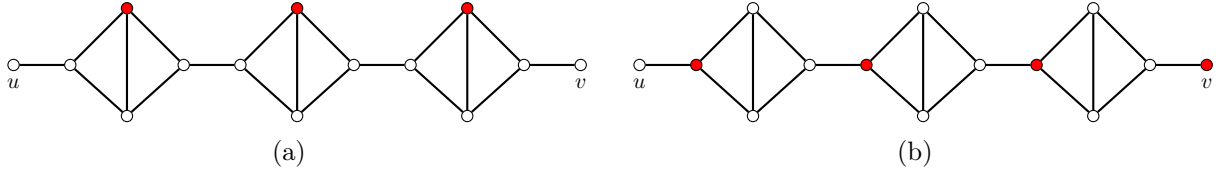
\begin{figure}[ht]
\centering

\begin{subfigure}{0.48\textwidth}
\centering
\begin{tikzpicture}[scale=0.75, transform shape,
    vertex/.style={circle, draw, fill=white, inner sep=2pt},
    redvertex/.style={circle, draw, fill=red, inner sep=2pt},
    edge/.style={draw, thick}
]

\node[vertex, label=below:$u$] (u) at (0,0) {};
\node[vertex] (a1) at (1,0) {};
\draw[edge] (u) -- (a1);

\node[redvertex] (t1) at (2,1) {};
\node[vertex] (b1) at (2,-0.9) {};
\node[vertex] (c1) at (3,0) {};
\draw[edge] (a1) -- (t1);
\draw[edge] (a1) -- (b1);
\draw[edge] (b1) -- (c1);
\draw[edge] (c1) -- (t1);
\draw[edge] (t1) -- (b1);

\node[vertex] (a2) at (4,0) {};
\draw[edge] (c1) -- (a2);

\node[redvertex] (t2) at (5,1) {};
\node[vertex] (b2) at (5,-0.9) {};
\node[vertex] (c2) at (6,0) {};
\draw[edge] (a2) -- (t2);
\draw[edge] (a2) -- (b2);
\draw[edge] (b2) -- (c2);
\draw[edge] (c2) -- (t2);
\draw[edge] (t2) -- (b2);

\node[vertex] (a3) at (7,0) {};
\draw[edge] (c2) -- (a3);

\node[redvertex] (t3) at (8,1) {};
\node[vertex] (b3) at (8,-0.9) {};
\node[vertex] (c3) at (9,0) {};
\draw[edge] (a3) -- (t3);
\draw[edge] (a3) -- (b3);
\draw[edge] (b3) -- (c3);
\draw[edge] (c3) -- (t3);
\draw[edge] (t3) -- (b3);

\node[vertex, label=below:$v$] (v) at (10,0) {};
\draw[edge] (c3) -- (v);

\end{tikzpicture}
\caption{}
\end{subfigure}
\hfill
\begin{subfigure}{0.48\textwidth}
\centering
\begin{tikzpicture}[scale=0.75, transform shape,
    vertex/.style={circle, draw, fill=white, inner sep=2pt},
    redvertex/.style={circle, draw, fill=red, inner sep=2pt},
    edge/.style={draw, thick}
]

\node[vertex, label=below:$u$] (u) at (0,0) {};
\node[redvertex] (a1) at (1,0) {};
\draw[edge] (u) -- (a1);

\node[vertex] (t1) at (2,1) {};
\node[vertex] (b1) at (2,-0.9) {};
\node[vertex] (c1) at (3,0) {};
\draw[edge] (a1) -- (t1);
\draw[edge] (a1) -- (b1);
\draw[edge] (b1) -- (c1);
\draw[edge] (c1) -- (t1);
\draw[edge] (t1) -- (b1);

\node[redvertex] (a2) at (4,0) {};
\draw[edge] (c1) -- (a2);

\node[vertex] (t2) at (5,1) {};
\node[vertex] (b2) at (5,-0.9) {};
\node[vertex] (c2) at (6,0) {};
\draw[edge] (a2) -- (t2);
\draw[edge] (a2) -- (b2);
\draw[edge] (b2) -- (c2);
\draw[edge] (c2) -- (t2);
\draw[edge] (t2) -- (b2);

\node[redvertex] (a3) at (7,0) {};
\draw[edge] (c2) -- (a3);

\node[vertex] (t3) at (8,1) {};
\node[vertex] (b3) at (8,-0.9) {};
\node[vertex] (c3) at (9,0) {};
\draw[edge] (a3) -- (t3);
\draw[edge] (a3) -- (b3);
\draw[edge] (b3) -- (c3);
\draw[edge] (c3) -- (t3);
\draw[edge] (t3) -- (b3);

\node[redvertex, label=below:$v$] (v) at (10,0) {};
\draw[edge] (c3) -- (v);

\end{tikzpicture}
\caption{}
\end{subfigure}

\caption{The construction of from $P'$ in the last part of the proof of Theorem \ref{theorem:cubic-claw-free}. $(a)$ $u, v \notin P'$, $(b)$ $u \notin P'$ and $v \in P'$.}
\label{fig:cubic-claw-free-string}
\end{figure}

%% file: claw-free-cubic-family.tex
\begin{theorem}
    There is an inifnite family of bridgeless claw-free cubic graphs $\mathcal{G}$ such that, for each $G \in \mathcal{G}$, $\gamma(G) = \frac{5}{4}\rho(G)$.
\end{theorem}
\begin{proof}
    We define the graph called $A$ as in Figure \ref{}
    and create $3t$ copies of $A$, say $A_1, A_2, \ldots A_{3t}$ for some positive integer $t$. After that, we join the last two vertices of each copy to the first two vertices of the subsequent copy. See Figure \ref{}.
    Let $H$ the resulting graph.
    After that, we replace each vertex with a triangle, as in the third option of Lemma \ref{lemma:claw-free-characterization}, creating a bridgeless claw-free cubic
    graph $G$ on $n$ vertices.

    We will use the notations used in the proof of Theorem \ref{theorem:cubic-claw-free}.
    Note that in this case, the set $S$ (vertices that belong to a 2-cycle in $H$), is empty. 
    Next we show, $\rho(H) \geq 4t$.
    Indeed, we can choose one every two vertices of the  top of $H$ (the set  formed by the $a_i$ vertices in each copy). As there are $12t$ such vertices, this packing have size $4t$.
    Now we show that $\rho(H) \leq 4t$.
    Suppose by contradiction that there exists a packing $P$  with $|P| > 4t$.
    For any $i \in \{1,\ldots, t\}$, let $B_i$
    the set of vertices that consists on the  union of the vertices of $A_{3i-2}$, $A_{3i-1}$ and $A_{3i}$. By pigeonhole principle, there exists an $i$
    such that $|B_i \cap P| \geq 5$. It is clear that $|A \cap P| \leq 2$,
    and that, if $|A \cap P| = 2$, then
    $A \cap P$ consists on one vertex in $\{a_1,b_1\}$ and one vertex in 
    $\{a_4,b_4\}$. Thus, $|A_{3i-2} \cap P|= |A_{3i} \cap P|=2$, but then 
    $A_{3i-1} \cap P = \emptyset$, a contradiction. We conclude that 
    $\rho(H) = 4t$ and, by Claim \ref{claim:gammaG'leqfracn3-Q}, $\gamma(G) = |V(H)|-4t= 20t$.

\input{Figures/cubic-claw-free-A}

    We now show that $\rho(G)=16t$.
    For this, we divide the set of $a_i$ vertices of $H$ into paths of size $3$.
    Suppose we have a path $xyz$.
    Then we select the vertices of $t(x)$ and $t(z)$ that are adjacent to a vertices in $t(y)$. Note that these vertices are at distance $3$ in $G$. Repeating this operation for any such path, we obtain a packing of size $8t$ for the $a_i$ vertices.
    We can join this packing with an analog packing of size $8t$ for the $b_i$ vertices of $G$. Indeed, an $a_i$ vertex will be at distance at least $3$ from a $b_i$ vertex. Hence  $\rho(G) \geq 16t$.
    Now, suppose for a moment that there exists a packing $P$ in $G$ with $|P|> 16t$.
    As there are $8t$ paths of size 3 in $H$, this implies that there exists a path, say $xyz$ such that $t(x)$, $t(y)$ and $t(z)$ have vertices in $P$. But then,  $t(y) \cap P$ is at distance at most $2$ from some vertex in $(t(u)\cup t(v)) \cap P$, a contradiction.
\end{proof}

%% file: Figures/cubic-claw-free-A.tex
\begin{figure}[ht]
\centering

\begin{subfigure}{0.28\textwidth}
\centering

\begin{tikzpicture}[scale=1,
  every node/.style={circle, draw, fill=black, inner sep=2pt}]

\node[label=above:$a_1$] (a1) at (0,1) {};
\node[label=above:$a_2$] (a2) at (1,1) {};
\node[label=above:$a_3$] (a3) at (2,1) {};
\node[label=above:$a_4$] (a4) at (3,1) {};

\node[label=below:$b_1$] (b1) at (0,0) {};
\node[label=below:$b_2$] (b2) at (1,0) {};
\node[label=below:$b_3$] (b3) at (2,0) {};
\node[label=below:$b_4$] (b4) at (3,0) {};

\draw (a1)--(a2)--(a3)--(a4);
\draw (b1)--(b2)--(b3)--(b4);

\draw (a1)--(b1);
\draw (a4)--(b4);

\draw (a2)--(b3);
\draw (b2)--(a3);

\end{tikzpicture}

\caption{}
\label{fig:small}
\end{subfigure}
\hfill
\begin{subfigure}{0.65\textwidth}
\centering

\begin{tikzpicture}[scale=1,
  every node/.style={circle, draw, fill=black, inner sep=2pt}]


\draw (-0.25,0.82) rectangle (2.25,1.18);
\draw (2.75,0.82) rectangle (4.75,1.18);
\draw (5.25,0.82) rectangle (7.25,1.18);
\draw (7.75,0.82) rectangle (10.25,1.18);


\draw (-0.25,-0.18) rectangle (2.25,0.18);
\draw (2.75,-0.18) rectangle (4.75,0.18);
\draw (5.25,-0.18) rectangle (7.25,0.18);
\draw (7.75,-0.18) rectangle (10.25,0.18);


\node[fill=white] (a11) at (0,1) {};
\node (a12) at (1,1) {};
\node (a13) at (2,1) {};
\node[fill=white] (a14) at (3,1) {};

\node (b11) at (0,0) {};
\node (b12) at (1,0) {};
\node (b13) at (2,0) {};
\node (b14) at (3,0) {};

\draw (a11)--(a12)--(a13)--(a14);
\draw (b11)--(b12)--(b13)--(b14);

\draw (a11)--(b11);
\draw (a14)--(b14);

\draw (a12)--(b13);
\draw (b12)--(a13);


\node (a21) at (3.5,1) {};
\node (a22) at (4.5,1) {};
\node (a23) at (5.5,1) {};
\node (a24) at (6.5,1) {};

\node (b21) at (3.5,0) {};
\node (b22) at (4.5,0) {};
\node (b23) at (5.5,0) {};
\node (b24) at (6.5,0) {};

\draw (a21)--(a22)--(a23)--(a24);
\draw (b21)--(b22)--(b23)--(b24);

\draw (a21)--(b21);
\draw (a24)--(b24);

\draw (a22)--(b23);
\draw (b22)--(a23);


\node (a31) at (7,1) {};
\node (a32) at (8,1) {};
\node (a33) at (9,1) {};
\node (a34) at (10,1) {};

\node (b31) at (7,0) {};
\node (b32) at (8,0) {};
\node (b33) at (9,0) {};
\node (b34) at (10,0) {};

\draw (a31)--(a32)--(a33)--(a34);
\draw (b31)--(b32)--(b33)--(b34);

\draw (a31)--(b31);
\draw (a34)--(b34);

\draw (a32)--(b33);
\draw (b32)--(a33);


\draw (a13)--(a21);
\draw (a14)--(a22);

\draw (b13)--(b21);
\draw (b14)--(b22);

\draw (a23)--(a31);
\draw (a24)--(a32);

\draw (b23)--(b31);
\draw (b24)--(b32);


\draw (a11) .. controls +(0,0.8) and +(0,0.8) .. (a34);
\draw (b11) .. controls +(0,-0.8) and +(0,-0.8) .. (b34);

\node[fill=white] (a14) at (3,1) {};
\node[fill=white] (a23) at (5.5,1) {};
\node[fill=white] (a32) at (8,1) {};

\end{tikzpicture}

\caption{}
\label{fig:large}
\end{subfigure}

\caption{(a) The graph A. (b) The graph $H$ constructed when $t=1$. White vertices form a packing in $H$.}
\label{fig:construction}
\end{figure}
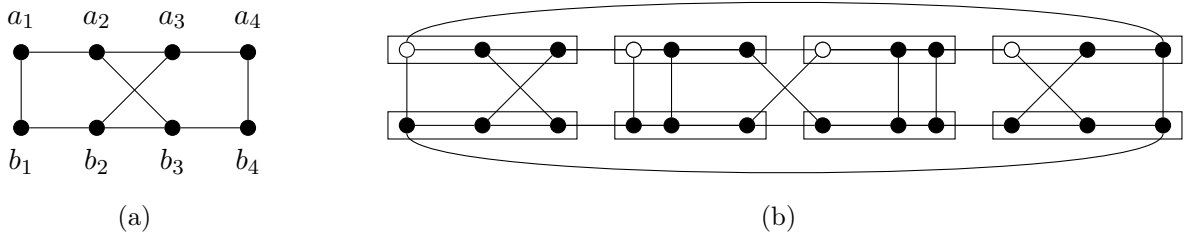

%% file: refer.bib
@article{DBLP:journals/corr/abs-2503-05562,
  author       = {Marthe Bonamy and
                  M{\'{o}}nika Csik{\'{o}}s and
                  Anna Gujgiczer and
                  Yelena Yuditsky},
  title        = {On graph classes with constant domination-packing ratio},
  journal      = {CoRR},
  volume       = {abs/2503.05562},
  year         = {2025},
  url          = {https://doi.org/10.48550/arXiv.2503.05562},
  doi          = {10.48550/ARXIV.2503.05562},
  eprinttype    = {arXiv},
  eprint       = {2503.05562},
  bibsource    = {dblp computer science bibliography, https://dblp.org}
}

@article{Favaron1999,
title = {Signed domination in regular graphs},
journal = {Discrete Mathematics},
volume = {158},
number = {1},
pages = {287-293},
year = {1996},
issn = {0012-365X},
doi = {https://doi.org/10.1016/0012-365X(96)00026-X},
url = {https://www.sciencedirect.com/science/article/pii/0012365X9600026X},
author = {Odile Favaron}
}

@article{Oum2011, 
title={Perfect Matchings in Claw-free Cubic Graphs}, volume={18}, 
url={https://www.combinatorics.org/ojs/index.php/eljc/article/view/v18i1p62},
number={1},
journal={The Electronic Journal of Combinatorics},
author={Oum, Sang-il},
year={2011},
month={Mar.},
pages={P62}}

@article{Henning2011,
title = {Dominating sets, packings, and the maximum degree},
journal = {Discrete Mathematics},
volume = {311},
number = {18},
pages = {2031-2036},
year = {2011},
author = {Michael A. Henning and Christian Löwenstein and Dieter Rautenbach}
}

@misc{bonamy2025,
      title={On graph classes with constant domination-packing ratio}, 
      author={Marthe Bonamy and Mónika Csikós and Anna Gujgiczer and Yelena Yuditsky},
      year={2025},
      eprint={2503.05562},
      archivePrefix={arXiv},
      primaryClass={math.CO},
      url={https://arxiv.org/abs/2503.05562}, 
}

@misc{ducz2026,
      title={Domination and packing in graphs}, 
      author={Ákos Dúcz and Anna Gujgiczer},
      year={2026},
      eprint={2602.18402},
      archivePrefix={arXiv},
      primaryClass={math.CO},
      url={https://arxiv.org/abs/2602.18402}, 
}

@article {BrandstadtCD98,
    AUTHOR = {Brandst\"{a}dt, A. and Chepoi, V. and Dragan, F.},
     TITLE = {The algorithmic use of hypertree structure and maximum
              neighbourhood orderings},
   JOURNAL = {Discrete Appl. Math.},
  FJOURNAL = {Discrete Applied Mathematics. The Journal of Combinatorial
              Algorithms, Informatics and Computational Sciences},
    VOLUME = {82},
      YEAR = {1998},
    NUMBER = {1-3},
     PAGES = {43--77},
      ISSN = {0166-218X},
   MRCLASS = {68R10 (05C85)},
  MRNUMBER = {1610005},
MRREVIEWER = {Gary MacGillivray},
       DOI = {10.1016/S0166-218X(97)00125-X},
       URL = {https://doi.org/10.1016/S0166-218X(97)00125-X}
}

@article {Farber84,
    AUTHOR = {Farber, M.},
     TITLE = {Domination, independent domination, and duality in strongly
              chordal graphs},
   JOURNAL = {Discrete Appl. Math.},
  FJOURNAL = {Discrete Applied Mathematics. The Journal of Combinatorial
              Algorithms, Informatics and Computational Sciences},
    VOLUME = {7},
      YEAR = {1984},
    NUMBER = {2},
     PAGES = {115--130},
      ISSN = {0166-218X},
   MRCLASS = {05C35},
  MRNUMBER = {727918},}

@article {LowensteinRR13,
    AUTHOR = {L\"{o}wenstein, C. and Rautenbach, D. and Regen,
              F.},
     TITLE = {Chiraptophobic cockroaches evading a torch light},
   JOURNAL = {Ars Combin.},
  FJOURNAL = {Ars Combinatoria},
    VOLUME = {111},
      YEAR = {2013},
     PAGES = {181--192},
      ISSN = {0381-7032},
   MRCLASS = {05C69},
  MRNUMBER = {3100171},
}

@article {MeirM75,
    AUTHOR = {Meir, A. and Moon, J.},
     TITLE = {Relations between packing and covering numbers of a tree},
   JOURNAL = {Pacific J. Math.},
  FJOURNAL = {Pacific Journal of Mathematics},
    VOLUME = {61},
      YEAR = {1975},
    NUMBER = {1},
     PAGES = {225--233},
      ISSN = {0030-8730},
   MRCLASS = {05C05},
  MRNUMBER = {401519},
MRREVIEWER = {Heiko Harborth},}

@article{DBLP:journals/dm/ClarkCJ90,
  author       = {Brent N. Clark and
                  Charles J. Colbourn and
                  David S. Johnson},
  title        = {Unit disk graphs},
  journal      = {Discret. Math.},
  volume       = {86},
  number       = {1-3},
  pages        = {165--177},
  year         = {1990},
  url          = {https://doi.org/10.1016/0012-365X(90)90358-O},
  doi          = {10.1016/0012-365X(90)90358-O},
  bibsource    = {dblp computer science bibliography, https://dblp.org}
}

@article{Henning2026,
author = {Henning, Michael and Maniya, Paras and Pradhan, Dinabandhu},
year = {2026},
month = {03},
pages = {},
title = {Domination number versus packing number in graphs},
journal = {Discussiones Mathematicae Graph Theory},
doi = {10.7151/dmgt.2624}
}

@article{Gomez2025,
title = {Domination and packing in graphs},
journal = {Discrete Mathematics},
volume = {348},
number = {5},
pages = {114393},
year = {2025},
issn = {0012-365X},
doi = {https://doi.org/10.1016/j.disc.2025.114393},
url = {https://www.sciencedirect.com/science/article/pii/S0012365X25000019},
author = {Renzo Gómez and Juan Gutiérrez}
}
